\documentclass[12pt,twoside]{amsart}
\usepackage{hyperref}
\usepackage{amsmath, amssymb, amscd, paralist,tabularx,supertabular,amsmath, verbatim, amsthm, amssymb, mathrsfs, manfnt,times,latexsym,amscd,graphicx, color, tikz}
\usepackage{amsmath, amsthm, amssymb}
\usepackage{paralist}  
\usepackage{manfnt}    
\usepackage{url}
\usepackage[all]{xy}

\usepackage{fullpage}

\DeclareMathAlphabet{\curly}{U}{rsfs}{m}{n}

\DeclareMathOperator{\PSL}{PSL}

\DeclareMathOperator{\SL}{SL}

\theoremstyle{definition}
\newtheorem*{df}{Definition}

\theoremstyle{remark}

\newtheorem{proposition}{Proposition}[section]
\def\1{\mathbf{1}}

\theoremstyle{remark}

\theoremstyle{plain}
\newtheorem{theorem}[proposition]{Theorem}
\newtheorem{cor}[proposition]{Corollary}
\newtheorem{defn}[proposition]{Definition}

\newtheorem{lemma}[proposition]{Lemma}

\def\C{\mathbf{C}}
\def\Q{\mathbf{Q}}

\def\Q{\mathbf{Q}}
\def\Z{\mathbf{Z}}

\def\1{\mathbf{1}}

\title{Constructing multi-cusped hyperbolic manifolds that are isospectral and not isometric}

\author{Benjamin Linowitz}
\address{Department of Mathematics\\ 
10 North Professor Street\\
Oberlin, OH 44074}
\email{benjamin.linowitz@oberlin.edu}

\begin{document}

\begin{abstract} 
In a recent paper Garoufalidis and Reid constructed pairs of $1$-cusped hyperbolic $3$-manifolds which are isospectral but not isometric. In this paper we extend this work to the multi-cusped setting by constructing isospectral but not isometric hyperbolic $3$-manifolds with arbitrarily many cusps. The manifolds we construct have the same Eisenstein series, the same infinite discrete spectrum and the same complex length spectrum. Our construction makes crucial use of Sunada's method and the Strong Approximation Theorem of Nori and Weisfeiler.
\end{abstract}

\maketitle 


\section{Introduction}\label{section:introduction}

In 1966 Kac \cite{Kac} famously asked ``Can one hear the shape of a drum?'' In other words, can one deduce the shape of a planar domain given knowledge of the frequencies at which it resonates? Long before Kac had posed his question mathematicians had considered analogous problems in more general settings and sought to determine the extent to which the geometry and topology of a Riemannian manifold is determined by its Laplace eigenvalue spectrum.

Early constructions of isospectral non-isometric manifolds include $16$-dimensional flat tori (Milnor \cite{Milnor}), compact Riemann surfaces (Vign{\'e}ras \cite{vigneras}) and lens spaces (Ikeda \cite{ikeda}). For an excellent survey of the long history of the construction of isospectral non-isometric manifolds we refer the reader to \cite{gordonsurvey1}.

In this paper we consider a problem posed by Gordon, Perry and Schueth \cite[Problem 1.2]{gordonsurvey2}: to construct complete, non-compact manifolds that are isospectral and non-isometric. This problem has received a great deal of attention in the case of surfaces. For example, Brooks and Davidovich \cite{Brooks} were able to use Sunada's method \cite{Sunada} in order to construct a number of examples of isospectral non-isometric hyperbolic $2$-orbifolds. For more examples, see \cite{gordonsurvey2}.

In a recent paper Garoufalidis and Reid \cite{GR} constructed the first known examples of isospectral non-isometric  $1$-cusped hyperbolic $3$-manifolds. The main result of this paper extends the work of Garoufalidis and Reid to the multi-cusped setting.

\begin{theorem}\label{maintheorem}
There exist finite volume orientable $n$-cusped hyperbolic $3$-manifolds that are isopectral and not isometric for arbitrarily large positive integers $n$.
\end{theorem}

Moreover, the manifolds we construct will be shown to have the same Eisenstein series, the same infinite discrete spectrum and the same complex length spectrum.

The author would like to thank Dubi Kelmer, Emilio Lauret, Ben McReynolds, Djordje Mili{\'c}evi{\'c}, Alan Reid and Ralf Spatzier for useful conversations concerning the material in this paper. The author is especially indebted to Jeff Meyer for his close reading of this paper and his many suggestions and comments. The work of the author is partially supported by NSF Grant Number DMS-1905437.

\section{Preliminaries}

Given a positive integer $d\geq 2$ we define ${\bf H}^d$ to be $d$-dimensional hyperbolic space, that is, the connected and simply connected Riemannian manifold of dimension $d$ having constant curvature $-1$. Let $\Gamma$ be a torsion-free discrete group of orientation preserving isometries of ${\bf H}^d$ such that the quotient space ${\bf H}^d/\Gamma$ has finite hyperbolic volume. Thus $M={\bf H}^d/\Gamma$ is a finite volume orientable hyperbolic $d$-manifold.

There exists a compact hyperbolic $d$-manifold $M'$ with boundary (possibly empty) such that the complement $M- M'$ consists of at most finitely many disjoint unbounded ends of finite volume, the {\it cusps} of $M$. Each cusp is homeomorphic to $N\times (0,\infty)$ where $N$ is a compact Euclidean $(d-1)$-manifold. 

Let $\Lambda$ denote the limit set of $\Gamma$ (i.e., the set of limit points of all the orbits of the action of $\Gamma$ on ${\bf H}^d$). A point $c\in\Lambda$ is called a {\it parabolic limit point} if it is the fixed point of some parabolic isometry $\gamma\in \Gamma$. The stabilizer $\Gamma_c < \Gamma$ of such a $c$ is called a {\it maximal parabolic subgroup} of $\Gamma$. A {\it cusp} of $\Gamma$ is a $\Gamma$-equivalence class of parabolic limit points and will be denoted by $[c]_\Gamma$. We will omit the subscript when the group is clear from context. The correspondence between cusps of $M$ and cusps of $\Gamma$ is given by the fact if $C$ is a cusp of $M$ then $C$ may be identified as $C=V_c/\Gamma_c$ where $V_c\subset {\bf H}^d$ is a precisely invariant horoball based at $c$ for some cusp $[c]$ of $\Gamma$.

\section{Spectrum of the Laplacian}

It is known that the space $L^2(M)$ has a decomposition \[L^2(M) = L^2_{disc}(M) \oplus L^2_{cont}(M)\] where $L^2_{disc}(M)$ corresponds to the discrete spectrum of the Laplacian on $M$ and $L^2_{cont}(M)$ corresponds to the continuous spectrum of $M$. The discrete spectrum of $M$ is a collection of eigenvalues $0\leq \lambda_1 \leq \lambda_2 \leq \cdots$ where each $\lambda_j$ occurs with a finite multiplicity. The continuous spectrum of $M$ is empty when $M$ is compact and otherwise is a union of finitely many intervals (one for each cusp of $M$) of the form \[\left[\frac{(d-1)^2}{4},\infty\right).\]

When $M$ is compact it is known that the discrete spectrum is infinite and obeys Weyl's Asymptotic Law. The precise analogue of Weyl's Asymptotic Law is in general not available when $M$ is not compact, though it is known in the case that $\Gamma$ is an arithmetic congruence group \cite{S, V77a, V79, V81}. 

The following elementary lemma will be useful in proving that certain manifolds have infinite discrete spectrum.

\begin{lemma}\label{finitecovers}
Let $M={\bf H}^d/\Gamma$ be a non-compact hyperbolic $d$-manifold and $M'={\bf H}^d /\Gamma'$ be a finite cover of $M$. If $M$ has an infinite discrete Laplace spectrum then so does $M'$.
\end{lemma}
\begin{proof}
The eigenfunctions associated to the discrete Laplace spectrum of $M$ are the set of eigenfunctions of the Laplacian that are invariant under $\Gamma$ and which are $L^2$-integrable over some (and hence any) fundamental domain for $\Gamma$. Any such function is also invariant under $\Gamma'$, and since the fundamental domain of $\Gamma'$ is a finite union of fundamental domains of $\Gamma$, the function will also be $L^2$ integrable over a fundamental domain for $\Gamma'$. It follows that $M'$ has an infinite discrete Laplace spectrum if $M$ does.\end{proof}

In order to discuss the spectrum of $M$ further we need to make clear the contribution of Eisenstein series. Let $[c]$ be a cusp of $\Gamma$ with stabilizer $\Gamma_c$. The {\it Eisenstein series} on $M$ associated to $[c]$ is defined to be the convergent series \[E_{M,c}(w,s) = \sum_{\gamma\in \Gamma_c \backslash \Gamma} y(\sigma^{-1}\gamma w)^s, \qquad w\in{\bf H}^d, s\in \C, \mathrm{Re}(s) > d-1, \] where $\gamma\in\Gamma$ represents a non-identity coset $\Gamma_c\gamma$ of $\Gamma_c$ in $\Gamma$ and $\sigma$ is the orientation preserving isometry of hyperbolic space taking the point at infinity to the cusp point $c$. Here we use the coordinates $z=(x,y)\in {\bf H}^d = {\bf R}^{d-1}\times {\bf R}^+$ for the upper half-space. 

Let $c_1,\dots, c_\kappa$ be representatives of a full set of inequivalent cusps of $\Gamma$. To ease notation we will temporarily refer to the Eisenstein series associated to the $i$-th cusp by $E_i(w,s)$. The constant term of $E_i(w,s)$ with respect to $c_j$ is denoted $E_{ij}(w,s)$ and satisfies \[E_{ij}(w,s) = \delta_{ij}y(\sigma_j^{-1}w)^s+\phi_{ij}(s)y(\sigma_j^{-1}w)^{d-1-s},\] where $\sigma_j$ is the orientation preserving isometry of hyperbolic space taking the point at infinity to the cusp point $c_j$ and where the coefficients $\phi_{ij}(s)$ define the {\it scattering matrix} $\Phi(s) = (\phi_{ij})$. We define the {\it scattering determinant} to be the function $\varphi(s) = \det \Phi(s)$. The Eisenstein series $E_j(w,s)$, the scattering matrix $\Phi(s)$ and the scattering determinant $\phi(s)$ have meromorphic extensions to the complex plane. The poles of $\varphi(s)$ are poles of the Eisenstein series and all lie in the half-plane $\mathrm{Re}(s) < \frac{d-1}{2}$, except for at most finitely many poles in the interval $(\frac{d-1}{2},d-1]$. The latter poles are related to the discrete spectrum as follows. Taking the residue of $E_j(w,s)$ at one of the latter poles yields an eigenfunction of the Laplacian with eigenvalue $s(d-1-s)$. The subset of the discrete spectrum arising from residues of poles of Eisenstein series (equivalently, of $\varphi(s)$) is called the {\it residual spectrum}. If $t$ is such a pole then we define the {\it multiplicity} at $t$ to be the order of the pole at $t$, plus the dimension of the eigenspace in the case when $t$ contributes to the residual spectrum as described above. This discussion motivates the following definition.

\begin{defn}
Let $M_1, M_2$ be $n$-cusped hyperbolic $d$-manifolds (for some positive integer $n$) of finite volume with scattering determinants $\varphi_1(s), \varphi_2(s)$. We say that $M_1$ and $M_2$ are isospectral if 
\begin{itemize}
\item $M_1$ and $M_2$ have the same discrete spectrum, counting multiplicities;
\item $\varphi_1(s)$ and $\varphi_2(s)$ have the same set of poles and multiplicities.
\end{itemize}
\end{defn}

The scattering determinant is in general very difficult to compute explicitly, although it has been worked out in several special case. For example, the scattering determinants associated to Hilbert modular groups over number fields have been computed in terms of Dedekind zeta functions by Efrat and Sarnak \cite{ES} and Masri \cite{Masri}.


\section{Cusps of finite covers of hyperbolic manifolds}\label{section:cusps}

We begin with a group theoretic lemma. Let $G$ be a group, $g$ be an element of $G$, and $H,K$ be subgroups of $G$. We define the double coset $HgK$ by \[HgK = \{ hgk : h\in H, k\in K\}.\]

\begin{lemma}\label{groupactiondoublecoset}
There is a bijection between the cosets of $H$ in $HgK$ and the cosets of $gKg^{-1}\cap H$ in $gKg^{-1}$.
\end{lemma}
\begin{proof}
Recall that $HgK$ is the union of the cosets $Hgk$ as $k$ varies over the elements of $K$. As right cosets of $H$ in $G$, two cosets $Hgk_1$ and $Hgk_2$ intersect if and only if they are equal. Observe that $Hgk_1=Hgk_2$ if and only if there is an element $h\in H$ such that $gk_1=hgk_2$, or equivalently, if and only if $k_1k_2^{-1}\in g^{-1}Hg$ (and thus is an element of $K\cap g^{-1}Hg$). This shows that $Hgk_1=Hgk_2$ if and only if $(K\cap g^{-1}Hg)k_1=(K\cap g^{-1}Hg)k_2$. We have therefore shown that 
the map $f$ given by $f(Hgk)=(K\cap g^{-1}Hg)k$ is a bijection between the cosets of $H$ in $HgK$ and of $K\cap g^{-1}Hg$ in $K$. We can now conjugate by $g$ to obtain a bijection between the cosets of $H$ in $HgK$ and the cosets of $(gKg^{-1}\cap H)$ in $gKg^{-1}$.
\end{proof}

Let $\Gamma$ be a discrete subgroup of $\mathrm{Isom}^+({\bf H}^d)$ and $x,y\in \partial{\bf H}^d$ be $\Gamma$-equivalent. Let $G$ be a subgroup of $\Gamma$ of finite index. We now define the set \[\Gamma_{x,y} = \{\gamma\in \Gamma : \gamma x\in G \cdot y\}.\]

\begin{lemma}\label{doublecosetequality}
There is an equality of sets $\Gamma_{x,y}=G\gamma P_x$, where $P_x=\mathrm{Stab}_\Gamma(x)$ and $\gamma$ is any element of $\Gamma$ such that $\gamma x=y$.
\end{lemma}
\begin{proof}
That any element of $G\gamma P_x$ lies in $\Gamma_{x,y}$ is clear. Suppose therefore that $\delta\in \Gamma_{x,y}$ and that $\delta x=gy=g(\gamma x)$. Then $(g\gamma)^{-1}\delta x=x$, hence $\gamma^{-1}g^{-1}\delta\in P_x$ and there exists $p\in P_x$ such that $\gamma^{-1}g^{-1}\delta=p$. This implies that $\delta=g\gamma p\in G\gamma P_x$ and completes the proof of the lemma.
\end{proof}

Let $M={\bf H}^d/\Gamma$ and $N={\bf H}^d/G$ be non-compact hyperbolic $d$-manifolds of finite volume and \[\pi: N \longrightarrow M\] be a covering. Let $c$ represent a cusp of $\Gamma$ and $P=\mathrm{Stab}_\Gamma(c)$.

\begin{df}
The preimage of a cusp of $M$ is always a union of cusps of $N$.  We say a cusp of $M$ {\it remains a cusp} of $N$ relative to $\pi$ when the preimage of that cusp has precisely one cusp of $N$. Algebraically, this is equivalent to $[c]_\Gamma=[c]_G$.
\end{df}

\begin{lemma}\label{cuspremainsacusp}
Suppose $c$ is a cusp representative of both $\Gamma$ and $G$ and that $[c]_\Gamma=[c]_G$. Then there is an equality of sets $\Gamma=GP$.
\end{lemma}
\begin{proof}
That $GP\subseteq \Gamma$ is clear as both $G$ and $P$ are subgroups of $\Gamma$. Now let $\gamma\in\Gamma$. Since $\Gamma c=Gc$ there exists an element $g\in G$ such that $\gamma c =g c$. It follows that $(g^{-1}\gamma)c=c$, hence $g^{-1}\gamma \in P$ and there exists $p\in P$ such that $g^{-1}\gamma=p$. This implies that $\gamma=gp$, concluding the proof.
\end{proof}

\begin{theorem}\label{cuspsstabs} Let $\{d_1,\dots, d_m\}$ represent the $G$-orbits on the elements of $\partial{\bf H}^d$ belonging to the cusp $[c]$ of $\Gamma$. Then 
\[ [\Gamma:G]=\sum_{i=1}^m [\mathrm{Stab}_\Gamma(d_i):\mathrm{Stab}_\Gamma(d_i) \cap G].\]
\end{theorem}
\begin{proof}

Write $\Gamma$ as a disjoint union of cosets $G\gamma_i$: \[\Gamma = \bigcup_{i=1}^r G\gamma_i.\] Since $\Gamma$ acts transitively on $[c]$, every element of $[c]$ is in the $G$ orbit of $\gamma_i d_1$ for some $i$. For each $j\in\{1,\dots,m\}$, fix $\delta_j\in\Gamma$ such that $\delta_j d_1=d_j$. By Lemma \ref{doublecosetequality}, $\Gamma_{d_1,d_j}=G\delta_j \mathrm{Stab}_\Gamma(d_1)$. Lemma \ref{groupactiondoublecoset} shows that $\Gamma_{d_1,d_j}$ is the union of $n$ cosets of $G$, where $n$ is the index of $\delta_j  \mathrm{Stab}_\Gamma(d_1) \delta_j^{-1} \cap G$ in $\delta_j  \mathrm{Stab}_\Gamma(d_1) \delta_j^{-1}$. As $\delta_j \mathrm{Stab}_\Gamma(d_1)\delta_j^{-1}=\mathrm{Stab}_\Gamma(\delta_jd_1)=\mathrm{Stab}_\Gamma(d_j)$, we see that $n=[\mathrm{Stab}_\Gamma(d_j):\mathrm{Stab}_\Gamma(d_j) \cap G]$. 

Putting all of this together, we see that $\Gamma$ is the disjoint union of $\Gamma_{d_1,d_j}$ as $j$ varies over $\{1,\dots, m\}$. Since each of these is the disjoint union of $[\mathrm{Stab}_\Gamma(d_j):\mathrm{Stab}_\Gamma(d_j) \cap G]$ cosets of $G$, we conclude that \[ [\Gamma:G]=\sum_{i=1}^m [\mathrm{Stab}_\Gamma(d_i):\mathrm{Stab}_\Gamma(d_i) \cap G],\] which completes our proof.\end{proof}

\begin{cor}\label{cuspsstaycusps}
We have an equality of indices $[\Gamma:G]=[\mathrm{Stab}_\Gamma(d):\mathrm{Stab}_\Gamma(d)\cap G]$ for all cusps $[d]$ of $G$ if and only if every cusp of $M$ remains a cusp of $N$.
\end{cor}

\begin{proof}

We first prove that if every cusp of $M$ remains a cusp of $N$  then $[\Gamma:G]=[\mathrm{Stab}_\Gamma(d):\mathrm{Stab}_G(d)]$ for all cusps $[d]$ of $G$. Fix a cusp $[d]$ of $G$ and define $P=\mathrm{Stab}_\Gamma(d)$. We must show that $[P:P\cap G]=[\Gamma:G]$. To that end, suppose that $p_1, p_2\in P$. Then 
\begin{align*}
Gp_1\cap Gp_2\neq \emptyset &\iff Gp_1=Gp_2 \\
&\iff p_1=gp_2 \text{ for some }g\in G \\
&\iff p_1p_2^{-1}=g \\
&\iff p_1p_2^{-1} \in P\cap G \\
&\iff (P\cap G)p_1=(P\cap G)p_2.
\end{align*}
We have therefore exhibited a bijection between the cosets of $G$ in $GP=\Gamma$ (the equality follows from Lemma \ref{cuspremainsacusp}) and the cosets of $(P\cap G)$ in $P$, hence $[\Gamma:G]=[P:P\cap G]$.

As the reverse direction is an immediate consequence of Theorem \ref{cuspsstabs}, our proof is complete.
\end{proof}

\begin{cor}\label{numberofcusps}
Suppose that $N$ is a normal cover of $M$. Let $[c]$ be a cusp of $\Gamma$ and $[d]$ be a cusp of $G$ contained in $[c]$. The number of cusps of $G$ contained in $[c]$ is \[\frac{[\Gamma:G]}{[\mathrm{Stab}_\Gamma(d):\mathrm{Stab}_\Gamma(d)\cap G]}.\]
\end{cor}
\begin{proof}
In light of Theorem \ref{cuspsstabs} it suffices to prove that if $[d_i],[d_j]$ are cusps of $G$ contained in the cusp $[c]$ of $\Gamma$ then $[\mathrm{Stab}_\Gamma(d_i):\mathrm{Stab}_\Gamma(d_i)\cap G]=[\mathrm{Stab}_\Gamma(d_j):\mathrm{Stab}_\Gamma(d_j)\cap G]$. To that end, let $\gamma\in\Gamma$ be such that $\gamma d_i=d_j$. Then \[\mathrm{Stab}_\Gamma(d_j) = \mathrm{Stab}_\Gamma(\gamma d_i)=\gamma \mathrm{Stab}_\Gamma(d_i) \gamma^{-1},\] hence, as $G=\gamma G \gamma^{-1}$, we have \[ [\mathrm{Stab}_\Gamma(d_j):\mathrm{Stab}_\Gamma(d_j)\cap G] = [\gamma \mathrm{Stab}_\Gamma(d_i)\gamma^{-1} :\gamma \mathrm{Stab}_\Gamma(d_i)\gamma^{-1}\cap \gamma G\gamma^{-1}]=[\mathrm{Stab}_\Gamma(d_i):\mathrm{Stab}_\Gamma(d_i)\cap G], \] which completes the proof.
\end{proof}


\section{Eisenstein series}\label{section:eisenstein}

\begin{theorem}\label{sameeis}
Let $M={\bf H}^d/\Gamma$ be a non-compact hyperbolic $d$-manifold and $N={\bf H}^d/G$ be a finite cover of $M$ with covering degree $n$. If a cusp $[c]$ of $\Gamma$ is also a cusp of $G$ (i.e., the preimage in $N$ of the corresponding cusp of $M$ is a single cusp) then $E_{M,c}(w,s)=E_{N,c}(w,s)$.
\end{theorem}
\begin{proof}
Let $c$ represent a fixed cusp of $\Gamma$ and $P=\mathrm{Stab}_\Gamma(c)$. We begin our proof by noting that Theorem \ref{cuspsstabs} shows that $[\Gamma:G]=[P:P\cap G]$, hence we may select a collection of coset representatives for $P\cap G$ in $P$ which is also a collection of coset representatives for $G$ in $\Gamma$. Let $\{\delta_1,\dots, \delta_n\}\subset P$ be such a collection.

An arbitrary term of $E_{M,c}(w,s)$ is of the form $y(\sigma^{-1}\gamma w)^s$ where $\gamma\in\Gamma$ represents a non-identity coset $P\gamma$ of $P$ in $\Gamma$ and $\sigma$ is the orientation preserving isometry of hyperbolic space taking the point at infinity to the cusp point $c$. Here we use the coordinates $z=(x,y)\in {\bf H}^d = {\bf R}^{d-1}\times {\bf R}^+$ for the upper half-space. Using our decomposition of $\Gamma$ into cosets of $G$ we see that there exists $\delta_j$ and $g\in G$ such that $\gamma=\delta_j g$. Because $\delta_j\in P$, the coset $P\gamma=P\delta_j g$ is equal to the coset $Pg$ as cosets of $P\backslash \Gamma$. In particular this implies that we may choose representatives for the cosets $P\backslash \Gamma$ to all lie in $G$. Note that for all $g_1,g_2\in G$ we have

\begin{align*}
Pg_1=Pg_2 &\iff g_1g_2^{-1}\in P \\
&\iff g_1g_2^{-1}\in P\cap G \\
&\iff (P\cap G)g_1=(P\cap G)g_2.
\end{align*}

It follows that \[E_{M,c}(w,s) = \sum_{\gamma\in P \backslash \Gamma} y(\sigma^{-1}\gamma w)^s = \sum_{g\in P\cap G \backslash G} y(\sigma^{-1}g w)^s = E_{N,c}(w,s).\]
\end{proof}

The following is an immediate consequence of Theorem \ref{sameeis}.

\begin{cor}
Suppose that $M$ is a cusped orientable finite volume hyperbolic $d$-manifold and that $M_1, M_2$ are finite covers of $M$ with the same covering degree and having the property that every cusp of $M$ remains a cusp of $M_i$ ($i=1,2$). Then all of the Eisenstein series of $M_1$ and $M_2$ are equal.
\end{cor}


\section{Congruence covers and $p$-reps}

Let $M$ be a non-compact finite volume orientable hyperbolic $3$-manifold. Let $c_1,\dots, c_\kappa$ represent a complete set of inequivalent cusps of $\pi_1(M)$ and $P_i$ be the subgroup of $\pi_1(M)$ that fixes $c_i$.

\begin{defn}
A surjective homomorphism $\rho: \pi_1(M) \rightarrow \PSL(2,p)$ is called a $p$-rep if, for all $i$, $\rho(P_i)$ is non-trivial and all non-trivial elements of $\rho(P_i)$ are parabolic elements of $\PSL(2,p)$.
\end{defn}

We remark that if $\rho: \pi_1(M) \rightarrow \PSL(2,p)$ is a $p$-rep then $\rho(P_i)$ must be a subgroup of $\PSL(2,p)$ of order $p$.
\begin{theorem}\label{parta}
Let $M$ be a $1$-cusped, non-arithmetic, finite volume orientable hyperbolic $3$-manifold with $p$-reps $\rho: \pi_1(M) \rightarrow \PSL(2,7)$ and $\rho': \pi_1(M) \rightarrow \PSL(2,11)$. Let $k$ be a number field with ring of integers $\mathcal O_k$ and degree not divisible by $3$. Assume that the faithful discrete representation of $\pi_1(M)$ can be conjugated to lie in $\PSL(2,\mathcal O_k)$. There exist infinitely many prime powers $q$ and covers $M_q$ of $M$ such that:

\begin{enumerate}
\item the composite homomorphism 
\[\rho_q:=\rho\circ \iota: \pi_1(M_q)\hookrightarrow \pi_1(M) \rightarrow \PSL(2,7)\] is a $p$-rep,
\item the degree over $M$ of the cover $M_q$ is $\frac{11}{2}(q^3-q)$, 
\item the number of cusps of $M_q$ is at least $q+1$, and
\item $M_q$ has an infinite discrete spectrum.
\end{enumerate}
\end{theorem}
\begin{proof}

We begin by constructing a finite cover $\widetilde{M}$ of $M$ which has an infinite discrete spectrum. The manifold $M_q$ will arise as a finite cover of $\widetilde{M}$ and will therefore have an infinite discrete spectrum by virtue of Lemma \ref{finitecovers}. To that end, let $H$ be an index $11$ subgroup of $\PSL(2,11)$. Such a subgroup is well-known to exist, and the cover of $M$ associated to the pullback subgroup of $H$ by $\rho'$ is a degree $11$ cover of $M$. Denote this cover by $\widetilde{M}$. We claim that $\widetilde{M}$ has one cusp. Let $P$ be the subgroup of $\pi_1(M)$ stabilizing the cusp of $M$. As was commented above, $\rho'(P)$ must be a cyclic subgroup of $\PSL(2,11)$ of order $11$. Since $H$ has index $11$ in $\PSL(2,11)$ and $|\PSL(2,11)|=660 = 2^2 \cdot 3 \cdot 5 \cdot 11$ it must be the case that $\rho'(P)\cap H$ is trivial. It follows that $[P:P\cap \pi_1(\widetilde{M})] = 11 = [\pi_1(M) : \pi_1(\widetilde{M})]$, hence $\widetilde{M}$ has one cusp by Corollary \ref{cuspsstaycusps}. It now follows from  \cite[Theorem 2.4]{GR} that $\widetilde{M}$  has an infinite discrete spectrum. We note that \cite[Theorem 2.4]{GR} has two hypotheses: that $\widetilde{M}$ be non-arithmetic and that $\widetilde{M}$ not be the minimal element in its commensurability class. That $\widetilde{M}$ is non-arithmetic is clear, since it is a finite cover of $M$, which is non-arithmetic. It is equally clear that $\widetilde{M}$ is not the minimal element of its commensurability class, since such an element cannot be a finite cover of another hyperbolic $3$-manifold.

We claim that $\pi_1(\widetilde{M})$ also admits a $p$-rep to $\PSL(2,7)$. In particular, we will show the homomorphism to $\PSL(2,7)$ obtained by composing the inclusion map $\pi_1(\widetilde{M})\hookrightarrow \pi_1(M)$ with $\rho : \pi_1(M) \rightarrow \PSL(2,7)$ is a $p$-rep. To see this, note that because $\gcd(11,|\PSL(2,7)|)=1$, the map $g\mapsto g^{11}$ is a bijection from $\PSL(2,7)$ to itself, hence our claim follows from the fact that for every $\gamma\in\pi_1(M)$ the element $\gamma^{11}$ lies in $\pi_1(\widetilde{M})$.

Given a proper, non-zero ideal $I$ of $\mathcal O_k$ we have a composite homomorphism \[\phi_I: \pi_1(\widetilde{M}) \longrightarrow \PSL(2,\mathcal O_k) \longrightarrow \PSL(2,\mathcal O_k/I)\] called the {\it level $I$ congruence homomorphism}. It follows from the Strong Approximation Theorem of Nori \cite{Nori} and Weisfeiler \cite{Weisfeiler} that for all but finitely many prime ideals $\mathfrak p$ of $\mathcal O_k$ the level $\mathfrak p$ congruence homomorphism $\phi_\mathfrak p$ is surjective.

By Dirichlet's Theorem on Primes in Arithmetic Progressions we may choose a prime $p$ satisfying $p\equiv 5 \pmod{168}$ which does not divide the discriminant of $k$. Let $\mathfrak p$ be a prime ideal of $\mathcal O_k$ lying above $p$ which has inertia degree $f$ satisfying $\gcd(f,3)=1$. Note that the existence of such a prime ideal $\mathfrak p$ follows from the well-known equality in algebraic number theory \[[k:\Q] = \sum_{i=1}^g e(\mathfrak p_i/p)f(\mathfrak p_i/p),\] where $p\mathcal O_k = \mathfrak p_1\cdots \mathfrak p_g$, $e(\mathfrak p_i/p)$ denotes the ramification degree of $\mathfrak p_i$ over $p$ and $f(\mathfrak p_i/p)$ denotes the inertia degree of $\mathfrak p_i$ over $p$. In particular our assertion follows from the hypothesis that $[k:\Q]$ not be divisible by $3$ and the fact that all of the ramification degrees $e(\mathfrak p_i/p)$ are equal to one (since $p$ doesn't divide the discriminant of $k$ and thus does not ramify in $k$).

We observed above that it follows from the Strong Approximation Theorem that for all but finitely many primes the associated congruence homomorphism is surjective. In light of our use of Dirichlet's Theorem on Primes in Arithmetic Progressions in the previous paragraph we may assume that $\mathfrak p$ was selected so that $\phi_\mathfrak p$ is surjective. Let $M_q$ be the cover of $\widetilde{M}$ associated to the kernel of $\phi_\mathfrak p$. The cover $M_q$ of $\widetilde{M}$ is normal of degree \[|\PSL(2,\mathcal O_k/\mathfrak p)| = |\PSL(2,p^f)| = \frac{p^{3f} - p^f}{2},\] which proves (ii) upon setting $q=p^f$.

Assertion (iii) follows from assertion (ii) and Corollary \ref{numberofcusps} since the image under $\phi_\mathfrak p$ of a cusp stabilizer $P_i$ will be an abelian subgroup of $\PSL(2,p^f)$ and thus will have order at most $\frac{p^f(p^f-1)}{2}$ by the classification of subgroups of $\PSL(2,q)$ (see \cite{Dickson}).

We now prove assertion (i). We will abuse notation and denote by $\rho$ the $p$-rep from $\pi_1(\widetilde{M})$ onto $\PSL(2,7)$. Because this $p$-rep was obtained by composing the inclusion of $\pi_1(\widetilde{M})$ into $\pi_1(M)$ with the $p$-rep from $\pi_1(M)$ onto $\PSL(2,7)$ (which was also denoted $\rho$), it suffices to prove assertion (i) with $\widetilde{M}$ in place of $M$. Let $N=\frac{p^{3f} - p^f}{2}=[\pi_1(\widetilde{M}):\pi_1(M_q)]$. As $\rho_q(\pi_1(M_q))$ contains $\rho_q(\gamma^N)=\rho(\gamma^N) =\rho(\gamma)^N$ for all $\gamma\in\pi_1(\widetilde{M})$ and $\rho:\pi_1(M)\rightarrow \PSL(2,7)$ is surjective, the surjectivity of $\rho_q$ follows from the fact (easily verifiable in SAGE \cite{SAGE}) that $\PSL(2,7)$ is generated by the $N$th powers of its elements whenever $p\equiv 5 \pmod{168}$ and $\gcd(f,3)=1$.

Let $P_0$ be the subgroup of $\pi_1(M_q)$ which fixes some cusp of $M_q$ and $P$ be the subgroup of $\pi_1(\widetilde{M})$ fixing the corresponding cusp of $\widetilde{M}$. Because $\rho: \pi_1(\widetilde{M}) \rightarrow \PSL(2,7)$ is a $p$-rep, $\rho(P)$ consists entirely of parabolic elements and therefore is a subgroup of $\PSL(2,7)$ of order $7$. Note that $[P:P_0]=d$ for some divisor $d$ of $N$. We will show that $N$, and thus $d$, is not divisible by $7$. Because $p$ was chosen so that $p\equiv 5 \pmod{168}$, we also have $p\equiv 5 \pmod{7}$ (since $168=2^3\cdot 3\cdot 7$). It is now an easy exercise in elementary number theory to show that $N=\frac{p^{3f} - p^f}{2}$ is not divisible by $7$ whenever $\gcd(f,3)=1$. Having shown that $\gcd(d,7)=1$, we observe that if $\gamma\in P$ has non-trivial image in $\PSL(2,7)$ then $\gamma^d\in P_0$ and thus $\rho_q(\gamma^d)=\rho(\gamma)^d$ is non-trivial in $\PSL(2,7)$.  Since $\rho_q(P_0)$ is a subgroup of $\rho(P)$ and thus also consists entirely of parabolic elements, this proves assertion (i).
\end{proof}


\section{Sunada's Method for constructing isospectral manifolds}

We begin this section by recalling the statement of Sunada's theorem \cite{Sunada}. 

Given a finite group $G$ with subgroups $H_1$ and $H_2$ we say that $H_1$ and $H_2$ are {\it almost conjugate} if, for all $g\in G$, \[\#(H_1\cap [g]) = \#(H_2\cap [g])\] where $[g]$ denotes the conjugacy class of $g$ in $G$.

\begin{theorem}[Sunada]
Let $M$ be a Riemannian manifold and $\rho: \pi_1(M) \rightarrow G$ be a surjective homomorphism. The coverings $M^{H_1}$ and $M^{H_2}$ of $M$ with fundamental groups $\rho^{-1}(H_1)$ and $\rho^{-1}(H_2)$ are isospectral.
\end{theorem}

The following is a group theoretic lemma of Prasad and Rajan \cite[Lemma 1]{PR} which they used to reprove Sunada's theorem. In what follows, if $G$ is a group and $V$ is a $G$-module then $V^G$ is the submodule of invariants of $G$.

\begin{lemma}\label{PRlemma}
Suppose that $G$ is a finite group with almost conjugate subgroups $H_1$ and $H_2$. Assume that $V$ is a representation space of $G$ over a field $k$ of characteristic zero. Then there exists an isomorphism $i: V^{H_1} \rightarrow V^{H_2}$, commuting with the action of any endomorphism $\Delta$ of $V$ which commutes with the action of $G$ on $V$; i.e. the following diagram commutes: 

\begin{center}
\begin{tikzpicture}[every node/.style={midway}]
  \matrix[column sep={4em,between origins}, row sep={2em}] at (0,0) {
    \node(A) {$V^{H_1}$}  ; & \node(C) {$V^{H_2}$}; \\
    \node(B) {$V^{H_1}$}; & \node (D) {$V^{H_2}$};\\
  };
  \draw[<-] (B) -- (A) node[anchor=east]  {$\Delta$};
  \draw[->] (B) -- (D) node[anchor=south] {$i$};
  \draw[->] (C) -- (D) node[anchor=west] {$\Delta$};
  \draw[->] (A) -- (C) node[anchor=north] {$i$};
\end{tikzpicture}
\end{center}
\end{lemma}

\begin{theorem}\label{partb}
Let $M={\bf H}^3/\Gamma$ be a cusped finite volume orientable hyperbolic $3$-manifold that is non-arithmetic and that is the minimal element in its commensurability class (i.e., $\Gamma=\mathrm{Comm}(\Gamma)$ where $\mathrm{Comm}(\cdot)$ denotes the commensurator). Let $M_0={\bf H}^3/\Gamma_0$ be a finite cover of $M$, $G$ be a finite group and $H_1,H_2$ be non-conjugate almost conjugate subgroups of $G$. Suppose that $\Gamma$ admits a homomorphism onto $G$ such that the induced composite homomorphism $\Gamma_0\hookrightarrow  \Gamma \rightarrow G$ is also onto. Let $M_1, M_2$ be the finite covers of $M_0$ associated to the pullback subgroups of $H_1$ and $H_2$ and assume that $M_1$ and $M_2$ both have the same number of cusps as $M_0$. Then $M_1$ and $M_2$ are are isospectral, have the same complex length spectra, are non-isometric and have infinite discrete spectra.
\end{theorem}
\begin{proof}

Our proof will largely follow the proof of the analogous result of Garoufalidis and Reid \cite[Theorem 3.1]{GR}.

We begin by proving that the manifolds $M_1$ and $M_2$ are non-isometric. Let $\Gamma_1,\Gamma_2$ be such that $M_1={\bf H}^3/\Gamma_1$ and $M_2={\bf H}^3/\Gamma_2$. If $M_1$ and $M_2$ are isometric then there exists $g\in\mathrm{Isom}({\bf H}^3)$ such that $g\Gamma_1 g^{-1}=\Gamma_2$. Such an element $g$ necessarily lies in the commensurator $\mathrm{Comm}(\Gamma)$ of $\Gamma$, and since $\Gamma=\mathrm{Comm}(\Gamma)$ we see that $g\in \Gamma$. By hypothesis there exists a surjective homomorphism $\rho: \Gamma \rightarrow G$. Projecting onto $G$ we see that $\rho(g) H_1 \rho(g)^{-1} = H_2$, which contradicts our hypothesis that $H_1$ and $H_2$ be non-conjugate.

To prove that $M_1$ and $M_2$ are isospectral we must show that their scattering determinants have the same poles with multiplicities and that they have the same discrete spectrum. Since $M_1$ and $M_2$ have the same covering degree over $M_0$, that their scattering determinants have the same poles with multiplicities follows immediately from Theorem \ref{sameeis}, which in fact shows that all of their Eisenstein series coincide. That $M_1$ and $M_2$ have the same discrete spectrum follows from Lemma \ref{PRlemma} with $k=\C$, $V=L^2_{disc}(M_0)$ and $\Delta$ the Laplacian.

That $M_1$ and $M_2$ have the same complex length spectra follows from the proof given by Sunada \cite[Section 4]{Sunada}.

That $M_1$ and $M_2$ have infinite discrete spectra follows from \cite[Theorem 2.4]{GR}.

\end{proof}


\section{Proof of Theorem \ref{maintheorem}}

In light of Theorems \ref{parta} and \ref{partb} it suffices to exhibit a non-arithmetic, $1$-cusped finite volume hyperbolic $3$-manifold $M$ which is the minimal element in its commensurability class and which admits $p$-reps onto $\PSL(2,7)$ and $\PSL(2,11)$. 

To prove this assertion, let $M$ be a hyperbolic $3$-manifold as in the previous paragraph and assume that $\pi_1(M)$ can be conjugated to lie in $\PSL(2,\mathcal O_k)$ for some number field $k$ whose degree is not divisible by $3$. (We will construct such a manifold below.) It follows from Theorem \ref{parta} that there exist infinitely many prime powers $q$ and covers $M_q$ of $M$ such that composing the inclusion $\pi_1(M_q)\hookrightarrow \pi_1(M)$ with the $p$-rep $\pi_1(M) \rightarrow \PSL(2,7)$ yields a $p$-rep and such that $M_q$ has at least $q+1$ cusps. 

We have seen that there is a surjective homomorphism $\rho: \pi_1(M_q)\rightarrow \PSL(2,7)$. It is well known that $\PSL(2,7)$ contains a pair of non-conjugate, almost conjugate subgroups of index $7$. Call these subgroups $H_1$ and $H_2$ and observe that since $|PSL(2,7)|=168$,  it must be that $H_1$ and $H_2$ have order $24$. Let $M_i = {\bf H}^3/\Gamma_i$ ($i=1,2$) be the manifold covers of $M_q$ associated to $H_1$ and $H_2$. 

Fix $i\in \{1,2\}$ and let $[d]$ be a cusp of $\Gamma_i$. Let $P_i=\mathrm{Stab}_{\Gamma_i}(d)$ and $P=\mathrm{Stab}_{\pi_1(M_q)}(d)$. Because the homomorphism $\rho: \pi_1(M_q)\rightarrow \PSL(2,7)$ is a $p$-rep, $\rho(P)$ is a cyclic subgroup of $\PSL(2,7)$ of order $7$. Since $H_i$ has order $24$ it must be that $\rho(P)\cap H_i$ is trivial. In particular it follows that $\rho(P_i)=1$ and consequently that $[\pi_1(M_q):\Gamma_i] = 7 = [P:P_i]$. Corollary \ref{cuspsstaycusps} now implies that every cusp of $M_q$ remains a cusp of $M_i$. In particular this shows that $M_1$ and $M_2$ both have the same number of cusps as $M_q$, and this number can be made arbitrarily large by taking the prime power $q$ (from Theorem \ref{parta}) to be arbitrarily large. Theorem \ref{maintheorem} now follows from Theorem  \ref{partb}.

We now construct a non-arithmetic, $1$-cusped finite volume hyperbolic $3$-manifold $M$ which is the minimal element in its commensurability class and which admits $p$-reps onto $\PSL(2,7)$ and $\PSL(2,11)$. We will additionally show that $\pi_1(M)$ can be conjugated to lie in $\PSL(2,\mathcal O_k)$ where $k$ is a number field of degree $8$.

\begin{figure}[h!]
  \includegraphics[width=2in]{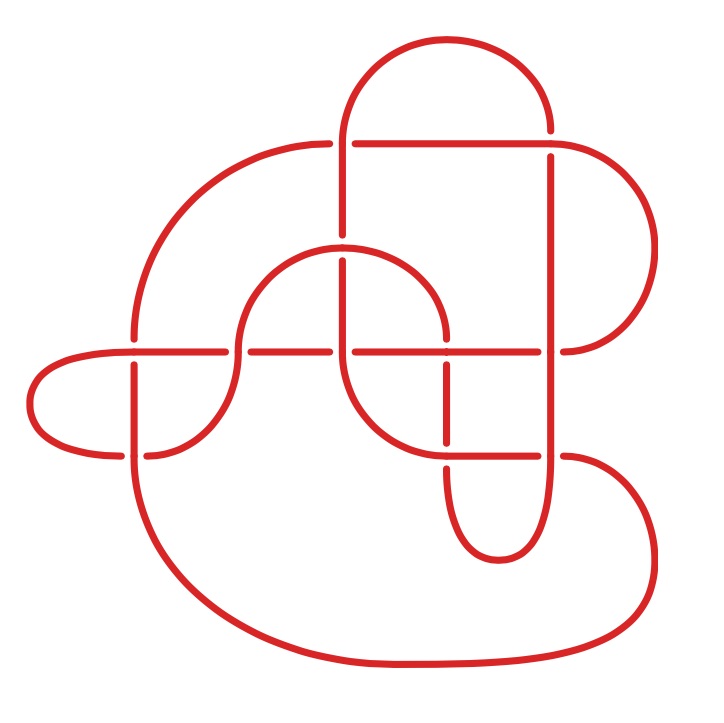}
  \caption{The knot K11n116.}
  \label{fig:k11n116}
\end{figure}

To that end, let $K$ be the knot K11n116 of the Hoste-Thistlethwaite table shown in Figure 1. The manifold $M=S^3\setminus K={\bf H}^3/\Gamma$ has $1$ cusp, volume $7.7544537602\cdots$ and invariant trace field $k=\Q(t)$ where $t=0.00106 + 0.9101192i$ is a root of the polynomial $x^8 - 2x^7 - x^6 + 4x^5 - 3x^3 + x + 1$. It was proven in \cite{GHH} that $M$ is the minimal element in its commensurability class (i.e., that $\Gamma = \mathrm{Comm}(\Gamma)$ where $\mathrm{Comm}(\Gamma)$ denotes the commensurator of $\Gamma$). The work of Margulis \cite{M} shows that this implies $M$ must be non-arithmetic. Moreover, a computation in Snap \cite{Snap} shows that $\Gamma$ has presentation \[\Gamma = \langle a,b,c \text{ } | \text{ }aaCbAccBB, \text{ }aacbCbAAB\rangle,\] and peripheral structure \[\mu = CbAcb, \qquad \lambda = AAbCCbacb.\] Here $A=a^{-1}, B=b^{-1}, C=c^{-1}$. In terms of matrices, we may represent $\Gamma$ as a subgroup of $\PSL(2,\mathcal O_k)$ via 

\[ a = \begin{pmatrix}-t^2+t-1 & t^7-3t^6+4t^5-t^4+t^2-t \\ -t^2+t-1  & 0\end{pmatrix},\]

\[ b = \begin{pmatrix}-t^7+2t^6-2t^5-3t^3+2t^2-3t-1 & t^6-2t^5+t^4+3t^3-2t^2+3t+2 \\ -t^7+3t^6-5t^5+4t^4-4t^3+2t^2-2t-1 & t^7-3t^6+5t^5-4t^4+4t^3-t^2+t+2\end{pmatrix},\]

and 

\[ c = \begin{pmatrix} -t^6+4t^5-8t^4+7t^3-5t^2-t & -2t^7+7t^6-14t^5+15t^4-12t^3+t^2+3t-1\\ t^5-3t^4+4t^3-3t^2+t & -t^7+4t^6-9t^5+11t^4-9t^3+3t^2+t-2\end{pmatrix}.\]

We now show that $\Gamma$ admits $p$-reps onto $\PSL(2,7)$ and $\PSL(2,11)$. We begin by exhibiting the $p$-rep onto $\PSL(2,7)$. As the discriminant of $k$ is $156166337$, which is not divisible by $7$, we see that $7$ is unramified in $k/\Q$. Using SAGE \cite{SAGE} we find that $7\mathcal O_k = \mathfrak p_1 \mathfrak p_2\mathfrak p_3,$ where the inertia degrees of the $\mathfrak p_i$ are $1,2,5$. We note that the prime $\mathfrak p_1$ of norm $7$ is equal to the principal ideal $(t-1)$. Upon identifying $\mathcal O_k /\mathfrak p_1$ with ${\bf F}_7$ we obtain a homomorphism from $\Gamma$ to $\PSL(2,7)$ by reducing the matrix entries of $a,b,c$ modulo $\mathfrak p_1$. The images of $a,b,c$ in $\PSL(2,7)$ are represented by
\[a=\begin{pmatrix}6 & 1 \\ 6 & 0\end{pmatrix}, \qquad b=\begin{pmatrix}1 & 6 \\ 3 & 5\end{pmatrix}, \qquad c=\begin{pmatrix}3 & 4 \\ 0 & 5\end{pmatrix},\] while the images of $\mu,\lambda$ in $\PSL(2,7)$ are represented by the parabolic matrices
\[\mu = \begin{pmatrix}0 & 4 \\ 5 & 5\end{pmatrix},\qquad \lambda = \begin{pmatrix}2 & 5 \\ 1 & 3\end{pmatrix}.\] It remains only to show that the homomorphism we have defined, call it $\rho_7$, is surjective. Our proof of this will make use of the following easy lemma.

\begin{lemma}\label{gens}
Let $p$ be a prime. The group $\SL(2,p)$ is generated by the matrices \[T=\begin{pmatrix}1 & 1 \\ 0 & 1\end{pmatrix}, \qquad U=\begin{pmatrix}1 & 0 \\ 1 & 1\end{pmatrix}.\]
\end{lemma}
\begin{proof}
The lemma follows from the fact that $\SL(2,\Z)$ is generated by the matrices in the lemma's statement. To see this, note that the usual generators of $\SL(2,\Z)$ are \[S=\begin{pmatrix}0 & -1 \\ 1 & 0\end{pmatrix},\qquad T=\begin{pmatrix}1 & 1 \\ 0 & 1\end{pmatrix},\] and $S=T^{-1}UT^{-1}$.\end{proof}

Surjectivity of our homomorphism $\rho_7: \Gamma \rightarrow \PSL(2,7)$ now follows from the fact that \[\begin{pmatrix}1 & 1 \\ 0 & 1\end{pmatrix}=\rho_7(b)^{-1}\rho_7(a)^{-2}\rho_7(b)^{-1}\rho_7(a)\rho_7(b)^{-1}\] and \[\begin{pmatrix}1 & 0 \\ 1 & 1\end{pmatrix}=\rho_7(c)\rho_7(a)^{-1}\rho_7(b)\rho_7(c)^2.\]

We have just shown that $\Gamma$ admits a $p$-rep onto $\PSL(2,7)$. We now show that $\Gamma$ admits a $p$-rep onto $\PSL(2,11)$ as well. In $k$ we have the factorization $11\mathcal O_k = \mathfrak p_1\mathfrak p_2\mathfrak p_3$ where the inertia degrees of the $\mathfrak p_i$ are $1,1,6$. We may assume without loss of generality that $\mathfrak p_1 = (t-4)$. Identifying $\mathcal O_k/\mathfrak p_1$ with ${\bf F}_{11}$ we see that the images in $\PSL(2,11)$ of $a,b,c$ are represented by the matrices \[a=\begin{pmatrix}9 & 6 \\ 9 & 0\end{pmatrix}, \qquad b=\begin{pmatrix}4 & 3 \\ 1 & 1\end{pmatrix}, \qquad c=\begin{pmatrix}10 & 1 \\ 6 & 4\end{pmatrix},\] while the images of $\mu,\lambda$ in $\PSL(2,11)$ are represented by the parabolic matrices
\[\mu = \begin{pmatrix}10 & 10 \\ 10 & 10\end{pmatrix},\qquad \lambda = \begin{pmatrix}10 & 0 \\ 6 & 10\end{pmatrix}.\] Finally, we show that our homomorphism $\rho_{11}: \Gamma \rightarrow \PSL(2,11)$ is surjective by applying Lemma \ref{gens}. To that end we simply note that  \[\begin{pmatrix}1 & 1 \\ 0 & 1\end{pmatrix}=\rho_{11}(a)^{-1}\rho_{11}(b)\rho_{11}(c)^{-1}\] and \[\begin{pmatrix}1 & 0 \\ 1 & 1\end{pmatrix}=\rho_{11}(c)\rho_{11}(a)^{2}.\]

This completes the proof of Theorem \ref{maintheorem}.


\end{document}